\documentclass{amsart}
\usepackage{amscd}
\usepackage{tikz}
\usetikzlibrary{cd}
\usepackage[pagebackref]{hyperref}

\newcommand{\Q}{\mathbb Q}
\newcommand{\Z}{\mathbb Z}
\newcommand{\fa}{{\mathfrak a}}
\newcommand{\fA}{{\mathfrak A}}
\newcommand{\fb}{{\mathfrak b}}
\newcommand{\fp}{{\mathfrak p}}
\newcommand{\fP}{{\mathfrak P}}
\newcommand{\fh}{{\mathfrak h}}
\newcommand{\fm}{{\mathfrak m}}
\newcommand{\frf}{{\mathfrak f}}
\newcommand{\fl}{{\mathfrak 2}}
\newcommand{\eps}{\varepsilon}
\newcommand{\OO}{{\mathcal O}}
\newcommand{\hj}{\widehat{j}}
\newcommand{\hA}{\widehat{A}}
\newcommand{\hC}{\widehat{C}}
\newcommand{\hR}{\widehat{R}}
\newcommand{\la}{\langle}
\newcommand{\ra}{\rangle}
\newcommand{\sth}{\,|\,}
\newcommand{\be}{\overline{e}}
\newcommand{\bnu}{\overline{\nu}}
\newcommand{\Gal}{\operatorname{Gal}}
\newcommand{\Cl}{\operatorname{Cl}}
\newcommand{\too}{\longmapsto}
\newcommand{\lra}{\longrightarrow}
\newcommand{\Lra}{\Longrightarrow}
\newcommand{\leri}{\longleftrightarrow}
\newcommand{\cok}{\operatorname{cok}}
\newcommand{\im}{\operatorname{im}}
\newcommand{\art}{\overset{Art}{\longleftrightarrow}}
\newcommand{\gen}{{\operatorname{gen}}}
\newcommand{\cen}{{\operatorname{cen}}}
\newcommand{\Ram}{{\operatorname{Ram}}}
\newcommand{\ts}{\textstyle}

\title{Kuroda's Class Number Formula}
\author{Franz Lemmermeyer}
\address{M\"orikeweg 1, 73489 Jagstzell, Germany}
\email{franz.lemmermeyer@gmx.de}

\newtheorem{thm}{Theorem}
\newtheorem{prop}{Proposition}
\newtheorem{lemma}{Lemma}

\begin{document}

\maketitle

\section*{Introduction}
Let $k$ be a number field and $K/k$ a $V_4$-extension, i.e., a normal
extension with $\Gal(K/k) = V_4$, where $V_4$ is Klein's four-group.
$K/k$ has three intermediate fields, say $k_1$, $k_2$, and $k_3$. We
will use the symbol $N^i$ (resp. $N_i$) to denote the norm of $K/k_i$
(resp. $k_i/k$), and by a widespread abuse of notation we will apply
$N^i$ and $N_i$ not only to numbers, but also to ideals and ideal
classes. The unit groups (groups of roots of unity, , groups of
fractional ideals, class numbers) in these fields will be denoted by
$E_k$, $E_1$, $E_2$, $E_3$, $E_K$ ($W_k$, $W_1$, \ldots,
$J_K$, $J_1$, \ldots, $h_k$, $h_1$, \ldots) respectively, and
the (finite) index $q(K) = E_K : E_1E_2E_3)$ is called the
{\em unit index} of $K/k$.

For $k = \Q$, $k_1 = \Q(\sqrt{-1}\,)$ and $k_2 = \Q(\sqrt{m}\,)$
it was already known to Dirichlet\footnote{Eisenstein [{\em \"Uber 
die Anzahl der quadratischen Formen in den verschiedenen complexen 
Theorieen}, J. Reine Angew. Math. {\bf 27} (1844), 311--316; 
Mathematische Werke I, Chelsea, New York, 1975, 89--94]
proved a similar formula for $K = \Q(\sqrt{-3},\sqrt{m}\,)$.}
\cite{Dir} that $h_K = \frac12 q(K) h_2h_3$.  Bachmann \cite{Ba},
Amberg \cite{Am} and Herglotz \cite{Her} generalized this class
number formula gradually to arbitrary extensions $K/\Q$ whose
Galois groups are elementary abelian $2$-groups.  A remark of
Hasse \cite[p. 3]{Has} seems to suggest\footnote{I have
meanwhile had a chance to verify Hasse's claim.} that
Varmon \cite{Var} proved a class number formula for extensions
with $\Gal(K/k)$ an elementary abelian $p$-group; unfortunately,
his paper was not accessible to me. Kuroda \cite{Kur2} later gave
a formula in case there is no ramification at the infinite primes.
Wada \cite{Wad} stated a formula for $2$-extensions of $k = \Q$
without any restriction on the ramification (and without proof),
and finally Walter \cite{Wal} used Brauer's class number relations
to deduce the most general Kuroda-type formula.

As we shall see below, Walter's formula for $V_4$-extensions does
not always give correct results if $K$ contains the $8$th roots
of unity. This does not, however, seem to effect the validity of
the work of Parry \cite{Par1,Par2} and Castela \cite{Cas}, both
of whom made use of Walter's formula.

The proofs mentioned above use analytic methods; for 
$V_4$-extensions $K/\Q$, however, there exist algebraic proofs 
given by Hilbert \cite{Hil} (if $\sqrt{-1} \in K$), Kuroda 
\cite{Kur1} (if $\sqrt{-1} \in K$), Halter-Koch \cite{HK} 
(if $K$ is imaginary), and Kubota \cite{Kub1,Kub2}. For base 
fields $k \ne \Q$, on the other hand, no non-analytic proofs 
seem to be known except for very special cases (see e.g. the 
very recent work of Berger \cite{Be}).

In this paper we will show how Kubota's proof can be generalized.
The proof consists of two parts; in the first part, where we
measure the extent to which $\Cl(K)$ is generated by classes
coming from the $\Cl(k_i)$, we will use class field theory in its
ideal-theoretic formulation (see Hasse \cite{HasB} or Garbanati
\cite{Gar}). The second part of the proof is a somewhat lengthy
index computation.

\section{Kuroda's Formula}
For any number field $F$, let $\Cl_u(F)$ be the odd part of the
ideal class group of $F$,i.e., the direct product of the $p$-Sylow
subgroups of $\Cl(F)$ for all odd primes $p$. It was already
noticed by Hilbert that the odd part of $\Cl(F)$ behaves well in
$2$-extensions, and the following fact is a special case of a theorem
of Nehrkorn \cite{Neh} (this special case can also be found in
Kuroda \cite{Kur2} or Reichardt \cite{Rei}):
\begin{equation}
\Cl_u(K) \simeq \Big(\prod_{i=1}^3 \Cl_u(k_i)/\Cl_u(k)\Big)
         \times \Cl_u(k) \quad \text{for $V_4$-extensions}\ K/k.
\end{equation}
Here $\prod$ denotes the direct product. This simple formula allows
us to compute the structure of $\Cl_u(K)$; of course we cannot
expect a similar result to hold for $\Cl_2(K)$, mainly because of
the following two reasons:
\begin{enumerate}
\item Ideal classes of $k_i$ may become principal in $K$
      (capitulation), and this means that we cannot regard
      $\Cl_2(k_i)$ as a subgroup of $\Cl_2(K)$.
\item Even if they do not capitulate, ideal classes of subfields may
      coincide in $K$: consider a prime ideal $\fp$ that ramifies
      in $k_1$ and $k_2$; then the prime ideals above $\fp$ in $k_1$
      and $k_2$ will generate the same ideal class in $K$.
\end{enumerate}

Nevertheless there is a homomorphism
$$ j: \Cl(k_1) \times \Cl(k_2) \times \Cl(k_3) \lra \Cl(K) $$
defined as follows: let $c_i = [\fa_i]$ be the ideal class in
$k_i$ generated by $\fa_i$; then $\fa_i\OO_K$ is the ideal in
$\OO_K$ (the ring of integers in $K$) generated by $\fa_i$,
and it is obvious that $j(c_1,c_2,c_3) = [\fa_1\fa_2\fa_3\OO_K]$
is a well defined group homomorphism, and that moreover
$$ h(K) = \frac{\cok j}{\ker j} \cdot h_1h_2h_3. $$
In order to compute $h(K)$ we have to determine the orders of
the groups $\ker j$ and $\cok j = \Cl(K)/\im j$. This will be
done as follows:

\begin{prop}
Let $\hj$ be the restriction of $j$ to the subgroup
$$ \hC = \{(c_1, c_2, c_3) \sth N_1c_1 N_2c_2 N_3c_3 = 1\} $$
of the direct product $\Cl(k_1) \times \Cl(k_2) \times \Cl(k_2)$.
Then
\begin{equation}\label{E1.2}
 h_k \cdot \frac{\cok j}{\ker j}  = \frac{\cok \hj}{\ker \hj}.
\end{equation}
\end{prop}

Now Artin's reciprocity law, combined with Galois theory, gives
a correspondence $\art$ between subgroups of $\Cl(K)$ and subfields
of the Hilbert class field $K^1$ of $K$. We will find that
$\im \hj \art K_\gen$, the genus class field of $K$ with respect
to $k$, and then the well known formula of Furuta \cite{Fur} shows
\begin{equation}\label{E1.3}
  \# \cok \hj = (\Cl(K): \im \hj) = (K_\gen:k) =
   2^{d-2} h_k \frac{\prod e(\fp)}{(E_k:H)}, \end{equation}
where
\begin{itemize}
\item $d$ is the number of infinite places ramified in $K/k$;
\item $e(\fp)$ is the ramification index in $K/k$ of a prime
      ideal $\fp$ in $k$, and $\prod$ is extended over all
      (finite)\footnote{The contribution from the infinite primes
      is taken care of by the factor $2^d$.} prime ideals of $k$;
\item $H$ is the group of units in $E_k$ that are norm
      residues\footnote{A norm residue is an element of $k$ that
      is a local norm for $K/k$ everywhere.} in $K/k$.
\end{itemize}

The computation of $\# \ker \hj$ is a bit tedious, but in the end
we will find
\begin{equation}\label{E1.4}
\ts \# \ker \hj = 2^{v-1} h_k^2 \prod e(\fp) \cdot (H: E_k^2)/q(K),
\end{equation}
where $v = 1$ if $K = k(\sqrt{\eps},\sqrt{\eta}\,)$ with units
$\eps, \eta \in E_k$, and $v = 0$ otherwise.

If we collect these results, define $\kappa$ to be the
$\Z$-rank of $E_k$, and recall the formula
$(E_k : E_k^2) = 2^{\kappa+1}$, we obtain

\begin{thm}
Let $K/k$ we a $V_4$-extension of number fields. Then Kuroda's
class number formula holds:
\begin{equation}\label{E1.5}
  h(K) = 2^{d-\kappa - 2 - v} q(K) h_1 h_2 h_3 / h_k^2.
\end{equation}
\end{thm}

In particular,
$$ h(K) =  \begin{cases}
  \frac14 q(K) h_1 h_2 h_3 & \text{if $k = \Q$ and $K$ is real}, \\
  \frac12 q(K) h_1 h_2 h_3 & \text{if $k = \Q$ and $K$ is complex}, \\
  \frac14 q(K) h_1 h_2 h_3/h_k^2 & \text{if $k$ is a complex quadratic
               extension of $\Q$}.
  \end{cases} $$

\section{The proofs}\label{S2}
In order to prove (\ref{E1.2}), we define a homomorphism
$$ \nu: C = \Cl(K_1) \times \Cl(k_2) \times \Cl(k_3)
    \lra \Cl(k), \quad \nu(c_1,c_2,c_3) = N_1c_1 N_2c_2 N_3c_3.$$
If at least one of the extensions $k_i/k$ is ramified,\footnote{At
a finite or infinite prime.} we know $N_i \Cl(k_i) = \Cl(k)$ by
class field theory. If all the $k_i/k$ are unramified, the groups
$N_i \Cl(k_i)$ will have index $2 = (k_i:k)$ in $\Cl(k)$, and
they will be different since
$$ k_i/k \art N_i \Cl(k_i) $$
in this case. Therefore $\nu$ is onto, and putting
$\hC = \ker \nu$ we get an exact sequence
$\begin{CD} 1 @>>> \hC @>>> C @>>> \Cl(k) @>>> 1. \end{CD} $

Let $\hj$ be the restriction of $j$ to $\hC$; then the diagram
$$ \begin{CD}
   1 @>>> \hC @>>> C @>>> \Cl(k) @>>{\nu}> 1 \\
   @. @VV{\hj}V @VV{j}V @VVV @. \\
   1 @>>> \Cl(L) @>>> \Cl(L) @>>> 1 @. \end{CD} $$
is exact and commutes. The snake lemma gives us an exact
sequence
$$ \begin{CD} 1 @>>> \ker \hj @>>> \ker j @>>> \Cl(k)
                @>>> \cok \hj @>>> \cok j @>>> 1, \end{CD} $$
and this implies the index relation (\ref{E1.2}) we wanted to
prove.

Before we start proving (\ref{E1.3}), we define $K^{(2)}$
to be the maximal subextension of $K_\gen/k$  such that
$\Gal(K^{(2)}/k)$ is an elementary abelian $2$-group. Moreover,
we let $J_K$ (resp. $H_K$) denote the group of (fractional)
ideals (resp. principal ideals) of $K$.

\begin{prop}\label{P2.1}
To every subfield $F$ of the Hilbert class field $K^1$ of $K$
there is a unique ideal group $\fh_F$ such that
$H_K \subseteq \fh_F \subseteq J_K$. Under this correspondence,
$$ \Gal(K^1/F) \simeq \Cl(K)/(J_K/\fh_F) \simeq \fh_F/H_K, $$
and we find the following diagram of subextensions $F/k$ of $K/k$
and corresponding Galois groups $\Gal(K^1/F)$:
\begin{center}
\begin{tikzcd}
  K^1 \arrow[leftrightarrow]{r} \arrow[d,dash] & 1 \arrow[d,dash] \\
  K_\gen \arrow[leftrightarrow]{r}  \arrow[d,dash] & \im \hj \arrow[d,dash] \\
  K^{(2)} \arrow[leftrightarrow]{r} \arrow[d,dash] & \im j   \arrow[d,dash] \\
  K \arrow[leftrightarrow]{r} & \Cl(K) 
\end{tikzcd}  
\end{center}
\end{prop}

\begin{proof}
The correspondence $K^{(2)} \leri \im j$ will not be needed
in the sequel and is included only for the sake of completeness;
the main ingredient for a proof can be found in Kubota
\cite[Hilfssatz 16]{Kub2}.

Before we start proving $K_\gen \leri \im \hj$, we recall that
$K_\gen$ is the class field of $k$ for the ideal group
$N_{K/k} H_K^{(\fm)} \cdot H_{\fm}^{(1)}$ of the norm residues
modulo $\fm$, where the defining modulus $\fm$ is a multiple of
the conductor $\frf(K/k)$ (the notation is explained in Hasse
\cite{HasB} or Garbanati \cite{Gar}, the result can be found
in Scholz \cite{Sch2} or Gurak \cite{Gur}). The assertion of
Herz \cite[Prop. 1]{Herz} that $K_\gen$ is the class field
for $N_{K/k} H_K^{(m)}$ is faulty: one mistake in his proof
lies in the erroneous assumption that every principal ideal
of $K$ is the norm of an ideal from $K^1$. Although this is
true for prime ideals, it does not hold in general, as the
following simple counterexample shows: the Hilbert class
field of $K = \Q(\sqrt{-5}\,)$ is $K^1 = K(\sqrt{-1}\,)$, and
the principal ideal $(1+\sqrt{-5}\,)$ cannot be a norm from $K^1$
since the prime ideals above $(2,1+\sqrt{-5}\,)$ and
$(3,1+\sqrt{-5}\,)$ are inert in $K^1/K$. Moreover, contrary to
Herz's claim, not every ideal in the Hilbert class field of $K$
is principal: this is, of course, only true for ideals coming
from $K$.
\end{proof}

\subsection*{Proof of (\ref{E1.3})}
Our task now is to transfer the ideal group
$N_{K/k} H_K^{(\fm)} \cdot H_\fm^{(1)}$ in $k$, which is
defined modulo $\fm$, to an ideal group in $K$ defined
modulo $(1)$. To do this we need

\begin{prop}\label{P2.2}
For $V_4$-extensions $K/k$, the following assertions are
equivalent:
\begin{enumerate}
\item[(i)] $r \in k^\times$ is a norm residue in $K/k$ at
           every place of $k$;
\item[(ii)] $r \in k^\times$ is a (global) norm from $k_1/k$
           and $k_2/k$;
\item[(iii)] there exist $\alpha \in K^\times$ and
           $a \in k^\times$ such that $r = a^2 \cdot N_{K/k} \alpha$.
\end{enumerate}
\end{prop}

The elements of $N_{K/k} H_K^{(\fm)} \cdot H_\fm^{(1)}$ therefore
have the form $a^2 \cdot N_{K/k} \alpha$, where $a \in k$,
$\alpha \in K$, and $(\alpha) + \fm = (1)$. Using the
Verschiebungssatz we find that $K_\gen/K$ belongs to the group
$$ \fh_\gen = \{ \fa \in J_K \sth \fa + \fm = (1),
    N_{K/k} \fa \in N_{K/k} H_K^{(\fm)} \cdot H_\fm^{(1)}\}.$$
Now $N_{K/k} \fa = a \cdot N_{K/k} \alpha$ if and only if
$N_{K/k} (\fa/\alpha) = (a)$; we put $\fb = \fa/\alpha$
and claim that there are ideals $\fa_i$ in $k_i$ such that
$\fb = \fa_1 \fa_2 \fa_3$. We assume without loss of
generality that $\fb$ is an (integral) ideal in $\OO_K$.
We may also assume that no ideal lying in a subfield $k_i$
divides $\fb$. But then any $\fP \mid \fb$ necessarily has
inertial degree $1$, and no conjugate of $\fP$ divides $\fb$.
Writing $\fp^m \parallel \fb$ we deduce
$$ N_{K/k} \fP^m \parallel N_{K/k} \fb = (a^2), $$
and this implies $2 \mid m$.

Let $\sigma$, $\tau$ and $\sigma\tau$ denote the nontrivial
automorphisms of $K/k$ fixing the elements of $k_1$, $k_2$
and $k_3$, respectively; the identity
$$ 2 = 1 + \sigma + \tau + \sigma\tau - (1+\sigma\tau)\sigma $$
in $\Z[\Gal(K/k)]$ shows
$\fP^2 = N^1 \fP \cdot N^2 \fP \cdot (N^3 \fP)^{-\sigma}$,
and we are done.

Now $(a^2) = N_{K/k} \fb = N_{K/k} \fa_1\fa_2\fa_3 =
 (N_1 \fa_1 \cdot N_2 \fa_2 \cdot N_3 \fa_3)^2$, and
extracting the square root we obtain
$(a) = N_1 \fa_1 \cdot N_2 \fa_2 \cdot N_3 \fa_3$.

Conversely, all ideals $\fa = \fa_1 \fa_2 \fa_3$ with
$\fa + \fm = (1)$ and
$(a) = N_1 \fa_1 \cdot N_2 \fa_2 \cdot N_3 \fa_3$ lie in
$\fh_\gen$, and the same is true of all principal ideals
prime to $\fm$ since the class field $K_\fh$ corresponding
to $\fh$ is unramified if and only if
$H_K^{(\fm)} \subseteq \fh$. Therefore
$$ \fh_\gen = \{ \fa = \fa_1 \fa_2 \fa_3 \sth
   \fa + \fm = (1), \
   N_1 \fa_1 \cdot N_2 \fa_2 \cdot N_3 \fa_3 = (a)
   \ \text{for some}\ a \in k\} \cdot H_K^{(\fm)}, $$
and by removing the condition $\fa + \fm = (1)$, which
amounts to replacing $\fh_\gen$ by an equivalent ideal
group, we finally see
$$ \fh_\gen = \{ \fa = \fa_1 \fa_2 \fa_3 \sth
   N_1 \fa_1 \cdot N_2 \fa_2 \cdot N_3 \fa_3 = (a)
   \ \text{for some}\ a \in k\} \cdot H_K. $$
The corresponding class group is $J_K/\fh_\gen$, and this gives
$$ \Gal(K_\gen/K) \simeq \fh_\gen/H_K =
   \{c = c_1c_2c_3 \sth N_1 c_1 N_2 c_2 N_3 c_3 = 1\} = \hC.$$
Now (\ref{E1.3}) follows from Furuta's formula for the genus
class number.

\subsubsection*{Proof of Prop. \ref{P2.2}}
It remains to prove Prop. \ref{P2.2}; this result is due to
Pitti \cite{Pi1,Pi2,Pi3}, and similar observations have been
made by Leep \& Wadsworth \cite{LW1,LW2}. Our proof of (ii)
$\Lra$ (iii) goes back to Kubota \cite[Hilfssatz 14]{Kub1},
while (iii) $\Lra$ (i) has already been noticed by Scholz
\cite[p.  102]{Sch1}.

(i) $\Lra$ (ii) is just an application of Hasse's norm residue
theorem for cyclic extensions.

(ii) $\Lra$ (iii). Choose $\alpha_1 \in k_1$ and $\alpha_2 \in k_2$
with $N_1 \alpha_1 = N_2 \alpha_2 = r$. Since $\sigma\tau$ acts
non-trivially on $k_1$ and $k_2$, this implies
$(\alpha_1/\alpha_2)^{1+\sigma\tau} = 1$. Hilbert's Theorem 90
shows the existence of $\alpha \in K^\times$ such that
$\alpha_1/\alpha_2 = \alpha^{1-\sigma\tau}$. Now
$$ \alpha^{1-\sigma\tau} =
   \alpha^{1+\sigma}(\alpha^{1+\tau})^{-\sigma}
     \quad \text{and} \quad
   \alpha^{1+\sigma}/\alpha_1 = (\alpha^{1+\tau})^\sigma/\alpha_2
      \in k_1 \cap k_2 = k. $$
Put $a = \alpha^{1+\sigma}/\alpha_1$ and verify
$N_{K/k} \alpha = (\alpha^{1+\sigma})^{1+\tau} = ra^2$.

(iii) $\Lra$ (i) is a consequence of formula (9) in \S\,6 of
Part II of Hasse's Bericht \cite{HasB}, which says
$$ \Big(\frac{\beta\,,\, k_1k_2}{\fp}\Big) \ = \
   \Big(\frac{\beta\,,\, k_1}{\fp}\Big)
   \Big(\frac{\beta\,,\, k_2}{\fp}\Big). $$
Since $r = N_i(N^i \alpha)/a)$ for $i = 1, 2$, we see that
$r$ is a norm from $k_1$ and $k_2$, and Hasse's formula tells
us that $r$ is a norm residue in $k_1k_2 = K$ at every place.

Before we proceed with the computation of $\# \ker \hj$,
let us pause for a moment and look at Prop. \ref{P2.1} with
more care. The fact that $K_\gen$ is the class field of $k$
for the ideal group $N_{K/k} H_K^{(\fm)} \cdot H_{\fm}^{(1)}$
is well known for abelian $K/k$. Moreover, the principal
genus theorem of class field theory says that $K_\gen$ is
the class field of $K$ for the class group $\Cl(K)^{1-\sigma} =
\{c^{1-\sigma} \sth c \in \Cl(K)\}$ if $\Gal(K/k) = \la \sigma \ra$
is cyclic. If $K/k$ is abelian but not necessarily cyclic, the
class field $K_\cen$ for the class group $\{c^{1-\sigma} \sth
c \in \Cl(K), \sigma \in \Gal(K/k)\}$ is called the {\em central
class field}, and in general $K_\cen$ is strictly bigger than
$K_\gen$. A description of $K_\gen$ in terms of the ideal class
group of $K$ is unknown for non-cyclic $K/k$, and Prop. \ref{P2.1}
answers this question for the simplest non-cyclic group, the
four-group $V_4 \simeq (\Z/2\Z)^2$. For other non-cyclic groups,
this remains an open problem.

In the $V_4$-case, the fact that $\la c^{\sigma-1} \sth c \in \Cl(K),
\sigma \in \Gal(K/k) \ra \subseteq \im \hj$ can be verified
directly by noting that $c^{\sigma-1} = (c^\sigma)^{\sigma\tau+1}
\cdot (c^{-1})^{\tau+1} \in C_2 \times C_3$ is annihilated by $\nu$.

\subsection*{Proof of (\ref{E1.4})}

The calculation of $\# \ker \hj$ will be done in several steps.
We call an ideal $\fa_1$ in $k_1$ {\em ambiguous} if
$\fa_1^\tau = \fa_1$. An ideal class $c \in \Cl(k_1)$ is called
{\em ambiguous} if $c^\tau = c$, and {\em strongly ambiguous}
if $c = [\fa_1]$ for some ambiguous ideal $\fa_1$.  Let $A_i$
denote the group of strongly ambiguous ideal classes in $k_i$
($i = 1, 2, 3$). Then $A = A_1 \times A_2 \times A_3$ is a
subgroup of $C$, and $\hA = \hC \cap A$ is a subgroup of $\hC$.
The idea of the proof is to restrict $\hj$ (once more) from
$\hC$ to $\hA$ and to compute the kernel of this restriction
by using the formula for the number of ambiguous ideal classes.

In (\ref{E1.3}) we defined $H$ as the group of units in $E_k$
that are norm residues in $K/k$ at every place of $k$. Using
Prop. \ref{P2.2} we see that
$$ H = \{\eta \in E_k \sth \eta = N_i \alpha_i \ \text{for some}\
   \alpha_i \in k_i, i = 1, 2, 3\}. $$
Let $H_0 = E_1^N \cap E_2^N \cap E_3^N$ be the subgroup of $H$
consisting of those units that are relative norms of units for
every $k_i/k$. The computation of $\# \ker \hj$ starts with
the following observation:

\begin{lemma}\label{L2.3}
Let $j^*$ denote the restriction of $\hj$ to $\hA$; then
\begin{equation}\label{E2.3}
 \# \ker \hj = (H:H_0) \cdot \# \ker j^*.
\end{equation}
\end{lemma}

Postponing the proof of Lemma \ref{L2.3} for a moment,
let us see how this implies (\ref{E1.4}). Let
$R = \{\fa_1\fa_2\fa_3 \sth \fa_i \in I_i\ \text{is ambiguous in}\
 k_i/k\}$ and $R_\pi = R \cap H_K$; then
\begin{equation}\label{E2.4}
  \# \ker j^* = \# A / (R:R_\pi).
\end{equation}
Now the computation of $\# \ker \hj$ is reduced to the determination
of $(H:H_0)$ and $(R:R_\pi)$; let $t = \# \Ram(K/k)$ denote the
number of (finite) prime ideals of $k$ that ramify in $K$, and
let $\lambda$ denote the $\Z$-rank of $E_K$. We will prove
\begin{equation}\label{E2.5}
  (R:R_\pi) = 2^{t+\kappa-\lambda-2-v} q(K) h_k \cdot
      \frac{\prod (E_i^N : E_k^2)}{(H_0:E_k^2)}.
\end{equation}
The number $\# A_i$ of strongly ambiguous ideal classes in
$k_i/k$ is given by the well known formula (cf. Hasse
\cite[Teil Ia, \S 13]{HasB}):

\begin{lemma}\label{P2.6}
We have
\begin{equation}\label{E2.6}
  \# A_i = 2^{\delta_i - \kappa - 2} h_k \cdot (E_i^N:E_k^2),
\end{equation}
where $\delta_i$ denotes the number of (finite and infinite) places
in $k$ that are ramified in $k_i/k$.
\end{lemma}

Once we know how the $\delta_i$ are related to $t$, $\kappa$,
$\lambda$, etc., we will be able to deduce (\ref{E1.4}) from
(\ref{E2.4}) -- (\ref{E2.6}). To this end, let $t_i$ be the
``finite part'' of $\delta_i$, i.e., the number
$\Ram(k_i/k)$ of prime ideals in $k$ ramified in $k_i/k$, and
let $d_i$ denote the infinite part. Then $\delta_i = d_i + t_i$,
and
\begin{equation}\label{E2.7}
  2^{t_1+t_2+t_3} = 2^t \prod e(\fp), \quad
  2d = d_1 + d_2 + d_3, \quad \text{and} \quad
  \lambda - 4\kappa = 3-2d.
\end{equation}
Since $\# A = \prod \# A_i$, we obtain from (\ref{E2.4})
and (\ref{E2.6})
$$ \# A = 2^{\delta_1 + \delta_2 + \delta_3 - 3\kappa - 6}
    h_k^3 \cdot \prod (E_i^N : E_k^2); $$
dividing by (\ref{E2.5}) yields
$$ \# \ker j^* = 2^{t_1+t_2+t_3-t+d_1+d_2+d_3 + \lambda
   - 4\kappa - 4 + v} h_k^2 \cdot (H_0 : E_k^2)/q(K), $$
and using (\ref{E2.7}) we find
$$ \# \ker j^* = 2^{v-1} h_k^2 \prod e(\fp) \cdot
   (H_0 : E_k^2)/q(K). $$
Substituting this formula into equation (\ref{E2.3})
we finally obtain (\ref{E1.4}).

\begin{proof}[Proof of Lemma \ref{L2.3}]

In order to prove (\ref{E2.3}) let
$([\fa_1], [\fa_2], [\fa_3]) \in \ker \hj$; then
$\fa_1\fa_2\fa_3 = (\alpha)$ for some $\alpha \in K^\times$.
Since $(N_{K/k} \alpha) = (N_1\fa_1 \cdot N_2 \fa_2 \cdot
N_3 \fa_3)^2$ (equality of ideals in $\OO_k$) and because
$([\fa_1], [\fa_2], [\fa_3]) \in \hC$, there exists
$a \in k$ such that $(N_{K/k} \alpha) = (a)^2$. This shows
that $\eta = (N_{K/k} \alpha)/a^2$ is a unit in $E_k$, which
is unique mod $NE_K \cdot E_k^2$. Moreover, $\eta \in H$ since
$\eta = N_i((N^i \alpha)/a)$. Therefore
$$ \theta_0: \ker \hj \lra H/NE_K \cdot E_k^2, \quad
   ([\fa_1], [\fa_2], [\fa_3]) \too \eta NE_K \cdot E_k^2,$$
is a well defined homomorphism. We want to show that $\theta_0$
is onto: to this end, let $\eta \in H$; using Prop. \ref{P2.2} we
can find an $a \in k$ such that $N_{K/k}\alpha = \eta a^2$. In
the proof of Prop. \ref{P2.1} we have seen that an equation
$N_{K/k} \fa = (a)^2$ implies the existence of ideals $\fa_i$
in $k_i$ such that $\fa = \fa_1\fa_2\fa_3$. This gives
$(\alpha) = \fa_1\fa_2\fa_3$.

Now $(N_1\fa_1 \cdot N_2\fa_2 \cdot N_3\fa_3)^2 =
  (N_{K/k} \alpha) = (a)^2$ yields
$(a) = (N_1\fa_1 \cdot N_2\fa_2 \cdot N_3\fa_3)$, and
we have shown $\eta \in \im \theta_0$.

Since $\theta_0: \ker \hj \lra H/NE_K \cdot E_k^2$
is onto, the same is true for any homomorphism
$\ker \hj \lra H/H_0$ that is induced by an inclusion
$NE_K \cdot E_k^2 \subseteq H_0 \subseteq H$. Obviously,
the group $H_0 = E_1^N \cap E_2^N \cap E_3^N$ defined above
is such a group, and so $\theta: \ker \hj \lra H/H_0$ is onto.
An element $([\fa_1], [\fa_2], [\fa_3]) \in \ker \hj$ belongs
to $\ker \theta$ if and only if

$$ \fa_1\fa_2\fa_3 = (\alpha),  \quad
   (a) = N_1\fa_1 \cdot N_2\fa_2 \cdot N_3\fa_3, \quad
    (N_{L/K} \alpha)/a^2 = \eta \in H_0. $$

Let $\rho_i = (N^i\alpha)/a$; then $\fa_1^{1-\tau} = (\rho_1)$,
$\fa_2^{1-\sigma\tau} = (\rho_2)$, $\fa_3^{1-\sigma} = (\rho_3)$
and $N_i\rho_i = \eta \in H_0$. Writing $\eta = N_i\eps_i$,
where $\eps_i \in E_i$, and replacing $\rho_i$ by $\rho_i/\eps_i$,
we may assume that $N_i\rho_i = 1$. Hilbert's Theorem 90 shows
$\rho_1 = \beta_1^{1-\tau}$, $\rho_2 = \beta_2^{1-\sigma\tau}$,
and $\rho_3 = \beta_3^{1-\sigma}$ for some $\beta_i \in k_i$.
The ideals $\fb_i = \fa_u \beta_i^{-1}$ are ambiguous, and we
have $[\fb_i] = [\fa_i]$. This means that the ideal classes
$[\fa_i]$ are strongly ambiguous, and we conclude
$$ \ker \theta \subseteq
   \ker \hj \cap A_1 \times A_2 \times A_3 = \ker j^*. $$
If, on the other hand, $([\fa_1], [\fa_2], [\fa_3]) \in \ker \hj$
and if the ideals $\fa_i$ are ambiguous, then the
$\rho_i = (N^i \alpha)/a$ are units, and
$$ \eta = \theta([\fa_1], [\fa_2], [\fa_3]) = N_i\rho_i
   \in E_1^N \cap E_2^N \cap E_3^N = H_0. $$

We have seen that $\ker \theta = \ker j^*$, which shows
that the sequence
$$ \begin{CD}
  1 @>>> \ker j^* @>>> \ker \hj @>{\theta}>> H/H_0 @>>> 1
  \end{CD} $$
is exact; (\ref{E2.3}) follows at once.
\end{proof}

\subsubsection*{Proof of (\ref{E2.4})}
The proof of (\ref{E2.4}) will be done in two steps. First we
notice that $\im j^*$ consists of those ideal classes
in $j(\hC)$ that are generated by ambiguous ideals in $k_i/k$.
Define
\begin{eqnarray*}
   R  & = & \{ \fA \sth \fA = \fa_1 \fa_2 \fa_3, \fa_i \in J_i
  \ \text{ambiguous}\}, \\
  \hR & = & \{ \fA \sth \fA = \fa_1 \fa_2 \fa_3, \fa_i \in J_i
  \ \text{ambiguous},\ \nu([\fa_1], [\fa_2], [\fa_3]) = 1\},
\end{eqnarray*}
and let $\pi$ be the homomorphism mapping $\fA \in \hR \subseteq J_K$
to $[\fA] \in \Cl(K)$. Then $\pi: \hR \lra \im j^*$ is obviously
onto, and $\ker \pi = \hR \cap H_K$. But if $\rho \in K$ and
$(\rho) = \fa_1\fa_2\fa_3 \in \hR$, then
$$ (\rho)^2 = (\fa_1\fa_2\fa_3)^2 =
    (N\fa_1 \cdot N\fa_2 \cdot N\fa_3) = (r)$$
for some $r \in k$. This shows
$$ \ker \pi = \{(\rho) \sth \rho \in K, (\rho)^2 = (r) \
   \text{for some}\ r \in k\} = R_\pi,$$
therefore
$$ (\hR : R_\pi) = \# \im \pi = \# \im j^* = (\hA : \ker j^*),$$
which is equivalent to
\begin{equation}\label{E2.8}
 \# \ker j^* = \frac{\# \hA}{(\hR : R_\pi)}.
\end{equation}

The homomorphism $\nu: C \lra \Cl(k)$ defined at the beginning
of Section \ref{S2} sends
$([\fa_1], [\fa_2], [\fa_3]) \in A = A_1 \times A_2 \times A_3
  \subseteq C$ to $[\fa_1\fa_2\fa_3]^2 \in \Cl(k)$ (remember
that the square of an ambiguous ideal of $k_i/k$ is an ideal
in $\OO_k$), and we see that
$$ \begin{CD}
   1 @>>> \hA @>>> A @>{\nu}>> A_1^2A_2^2A_3^2 @>>> 1
   \end{CD} $$
is a short exact sequence. Now
$$ \begin{CD}
   1 @>>> \hR @>>> R @>{\bnu}>> A_1^2A_2^2A_3^2 @>>> 1,
   \end{CD} $$
where $\bnu(\fa_1\fa_2\fa_3) = \nu([\fa_1], [\fa_2], [\fa_3])
  = [\fa_1\fa_2\fa_3]^2$, is also exact.
From these facts we conclude that $(A:\hA) = (R:\hR)$, and
this allows us to transform (\ref{E2.8}):
$$ \# \ker j^* = \frac{\# \hA}{(\hR:R_\pi)} =
   \frac{(A:\hA) \cdot \# \hA}{(R:\hR)(\hR:R_\pi)} =
   \frac{\# A}{(R:R_\pi)}. $$
This is just (\ref{E2.4}).

Next we determine $(R:R_\pi)$. To this end, let $(\rho) \in R_\pi$.
Then $(\rho)^2 = (r)$ for some $r \in k^\times$, and $\eta = \rho^2/r$
is a unit in $\OO_K$. Since the ideal $(\rho)$ is fixed by
$\Gal(K/k)$, $\eta_i = (N^i \rho)/r$ is a unit in $E_i$. If
$\sigma \in \Gal(K/k)$ is an automorphism that acts nontrivially
on $k_3/k$, we find that $\eta = \eta_1\eta_2\eta_3^{-\sigma}
\in E_1E_2E_3$, where
$$ N_1 \eta_1 = N_2 \eta_2 = N_3 (\eta_3^{-\sigma})
   = (N_{K/k} \rho)/r^2. $$
The unit $\eta$ we have found is determined up to a factor
$\in E_kE^2$ (from now on, the unit group $E_K$ will appear
quite often, so we will write $E$ instead of $E_K$), and we
can define a homomorphism $\varphi: R_\pi \lra E/E_kE^2$ by
assigning the class of the unit $\eta = \rho^2/r$ to an
ideal $(\rho) \in R_\pi$ that satisfies $(\rho)^2 = (r)$,
$r \in k^\times$.  We cannot expect $\varphi$ to be onto because
only those units $\eta_1\eta_2\eta_3 \in E_1E_2E_3$ can
lie in the image of $\varphi$ whose norms $N_i \eta_i$ coincide.
Therefore we define
$$ E^* = \{ e_1e_2e_3 \sth e_i \in E_i,
  N_1e_1 \equiv N_2 e_2 \equiv N_3  e_3 \bmod E_k^2\} $$
and observe that $\im \varphi \subseteq E^*/E_kE^2$. Moreover,

\begin{lemma}\label{L2.9}
For $\eta = e_1e_2e_3 \in E^*$, the extension
$K(\sqrt{\eta}\,)/k$ is normal with elementary abelian
Galois group $\Gal(K(\sqrt{\eta}\,)/k)$, and there are
$\rho \in K^\times$ and $r \in k^\times$ such that
$\eta = \rho^2/r$.
\end{lemma}

\begin{proof}
$K(\sqrt{\eta}\,)/k$ is normal if and only if for every
$\sigma \in \Gal(K/k)$ there exists an $\alpha_\sigma \in K^\times$
such that $\eta^{1-\sigma} = \alpha_\sigma^2$. Let
$\Gal(K/k) = \{1,\sigma, \tau, \sigma\tau\}$ and suppose
that $\sigma$ fixes $k_1$; then
$$ \eta^{1-\sigma} = (e_1e_2e_3)^{1-\sigma}
   = (e_2e_3)^{1-\sigma} = (e_2e_3)^2/(N_2e_2 \cdot N_3e_3),$$
and this is a square in $K^\times$ since
$N_2e_2 \equiv N_3e_3 \bmod E_k^2$.

It is an easy exercise to show that $\Gal(K(\sqrt{\eta}\,)/k)$
is elementary abelian if and only if $\alpha_\sigma^{1+\sigma}
 = \alpha_\tau^{1+\tau} = \alpha_{\sigma\tau}^{1+\sigma\tau} = +1$.
In our case, these equations are easily verified (for example
$\alpha_\sigma = e_2e_3/e$ for some $e \in E_k$ such that
$e^2 = N_2e_2 \cdot N_3 e_3$, and therefore
$\alpha_\sigma^{1+\sigma} = (N_2e_2 \cdot N_3e_3)/e^2 = +1$).

Now $K(\sqrt{\eta}\,)/k$ is elementary abelian, and so
$k(\sqrt{\eta}\,) = k(\sqrt{r}\,)$ for some $r \in k^\times$.
This implies the existence of $\rho \in k^\times$ such that
$\rho^2 = \eta r$.
\end{proof}

Because of Lemma \ref{L2.9}, $\varphi: R_\pi \lra E^*/E_kE^2$
is onto. Moreover,
\begin{eqnarray*}
  \ker \varphi & = & \{ (\rho) \in R_\pi \sth
           \rho^2/r = ue^2, u \in E_k, e \in E\} \\
         & = & \{ (\rho) \in R_\pi \sth
           \exists\ r \in k^\times, e \in E: (\rho/e)^2 = r\} \\
         & = & \{ (\rho) \in R_\pi \sth
           \rho^2 = r\ \text{for}\ r \in k^\times \}.
\end{eqnarray*}
Let $R_0 = \ker \varphi$; the group of principal ideals $H_k$ is a
subgroup of $R_0$, and it has index $(R_0:H_k) = 2^{2-u}$, where
$2^u = (E^{(2)}:E_k)$ and $E^{(2)} = \{e \in E: e^2 \in E_k\}$.
The proof is very easy: let $\Lambda = \{\rho \in K^\times \sth
\rho^2 \in k^\times\}$ and map $\Lambda/k^\times$ onto $R_0/H_k$
by sending $\rho\cdot k^\times$ to $(\rho)\cdot H_k$. The sequence
$$ \begin{CD}
   1 @>>> E^{(2)}k^\times/k^\times @>>> \Lambda/k^\times
     @>>> R_0/H_k @>>> 1 \end{CD} $$
is exact, and since $\Lambda/k^\times$ has order $4$
($\Lambda/k^\times = \{k^\times, \sqrt{a} \cdot k^\times,
 \sqrt{b}\cdot k^\times, \sqrt{ab} \cdot k^\times \}$ for
$K = k(\sqrt{a}, \sqrt{b}\,)$) and
$E^{(2)} k^\times/k^\times \simeq E^{(2)}/E_k$, the claim
follows. We see
$$ (R_0 : H_k) = \begin{cases}
   1 & \text{if we can choose $a, b \in E_k$}, \\
   2 & \text{if we can choose $a \in E_k$ or $b \in E_k$,
             but not both}, \\
   4 & \text{otherwise}. \end{cases} $$

Now we find $(R : H_k) = (R:J_k)(J_k:H_k) = 2^t h_k$,
where $t = \# \Ram(K/k)$, and
$$ (R:R_\pi) \ = \ \frac{(R:H_k)}{(R_\pi:R_0)(R_0:H_k)}
      \ = \ 2^{t-2} h_k \frac{(E^{(2)}:E_k)}{(E^*:E_kE^2)}. $$
Since
\begin{eqnarray*}
  (E:E_kE^2)   & = & \frac{(E:E^2)}{(E_kE^2:E^2)}, \\
  (E_kE^2:E^2) & = & (E_k:E^2\cap E_k) \ = \
        \frac{(E_k:E^2)}{(E^2 \cap E_k : E_k^2)} \quad \text{and} \\
  (E^2 \cap E_k : E_k^2) & = & (E^{(2)} : E_k),
\end{eqnarray*}
we get $(E:E_kE^2) = 2^{\lambda-\kappa} (E^{(2)}:E_k)$, where
$\lambda$ and $\kappa$ denote the $\Z$-ranks of $E$ and $E_k$,
respectively. Collecting everything, we find
\begin{eqnarray*}
  (R:R_\pi) & = & 2^t h_k \frac1{(E^*:E_kE^2)(R_0:H_k)}
    \ = \ 2^t h_k \frac{(E:E^*)(E^{(2)}:E_k)}{4(E:E_kE^2)} \\
    & = & 2^{t+\kappa-\lambda-2} h_k (E:E^*).
\end{eqnarray*}
But $(E:E^*) = (E:E_1E_2E_3)(E_1E_2E_3:E^*)$, and the first
factor is the unit index $q(K)$; this shows
\begin{equation}\label{E2.10}
  (R:R_\pi) \ = \ 2^{t+\kappa-\lambda-2} q(K) h_k (E_1E_2E_3:E^*).
\end{equation}

In order to study the group $E_1E_2E_3/E^*$, we define
$E_i^* = \{e_i \in E_i \sth N_ie_i \in E_k^2\}$ and notice that
$E_1^*E_2^*E_3^* \subseteq E^* \subseteq E_1E_2E_3 \subseteq E$.
The group $E^*/E_1^*E_2^*E_3^*$ is actually one we have encountered
before:

\begin{lemma}
We have
\begin{equation}\label{E2.11}
 E^*/E_1^*E_2^*E_3^* \simeq H_0/E_k^2.
\end{equation}
\end{lemma}

\begin{proof}
Map $e_1e_2e_3 \in E^*$ to the coset
$N_1e_1 E_k^2 = N_2e_2 E_k^2 = N_3e_3 E_k^2$.
\end{proof}

It is therefore sufficient to compute the index
$(E_1E_2E_3 : E^*/E_1^*E_2^*E_3^*)$; to this end we introduce
the natural homomorphism
$$ \xi: E_1/E_1^* \times E_2/E_2^* \times E_3/E_3^*
    \lra E_1E_2E_3 / E_1^*E_2^*E_3^*, $$
which, of course, is onto. Letting $\be_1$ denote the coset
$e_i E_i^*$ we find
$$ \ker \xi = \{(\be_1, \be_2, \be_3): e_1e_2e_3 = u_1u_2u_3 \
   \text{for some}\ u_i \in E_i^*\}.$$
We need to characterize $\ker \xi$. Assume that
$(\be_1, \be_2, \be_3) \in \ker \xi$; then $e_1e_2e_3 = u_1u_2u_3$
for some $u_i \in E_i^*$. Replacing the $\be_i = e_i E_i^*$
by $e_iu_i^{-1} E_i^*$ if necessary we may assume that
$e_1e_2e_3 = 1$. Applying $1+\sigma$ to this equation (where
$\sigma$ fixes $k_1$) yields $e_1^2 N_2e_2 N_3e_3 = 1$, and
this implies $e_2^2 \in E_k$; in a similar way we find $e_2^2 \in E_k$
and $e_3^2 \in E_k$. If $N_2e_2$ were a square in $E_k$, so were
$N_3e_3$, and $e_1$ would have to lie in $E_k$: but then
$e_i \in E_i^*$ for $i = 1, 2, 3$, and $(\be_1, \be_2, \be_3)$
is trivial. So if $\ker \xi \ne 1$, we must have
$e_i \in E_i \setminus E_k$ for $i = 2, 3$; but we have seen
$e_i^2 =: \eps_i \in E_k$, so we get $k_i = k(\sqrt{\eps_i}\,)$
for $i = 2, 3$ and, therefore, $k_1 = k(\sqrt{\eps_2\eps_3}\,)$.
Moreover,
$$ \ker \xi = \{1,
(\sqrt{\eps_1}\cdot E_1^*, \sqrt{\eps_2}\cdot E_2^*,
 \sqrt{\eps_3}\cdot E_3^*)\} $$
in this case.

Thus we have shown that $\ker \xi \ne 1$ implies $u=2$ and
$\# \ker \xi = 2$, where the index $2^u = (E^{(2)}:E_k)$ was
introduced above. If, on the other hand, $u = 2$, then
$k_i = k(\sqrt{\eps_i}\,)$ for units $\eps_i \in E_k$, and
$(\sqrt{\eps_1}\cdot E_1^*, \sqrt{\eps_2}\cdot E_2^*,
 \sqrt{\eps_3}\cdot E_3^*)$ is a nontrivial element of
$\ker \xi$. Therefore $\# \ker \xi = 2^v$ with
$v = 2^u - u - 1$, and
\begin{equation}\label{E2.12}
  (E_1E_2E_3 : E_1^* E_2^* E_3^*) = 2^{-v} \prod (E_i : E_i^*).
\end{equation}

To determine $(E_i:E_i^*)$, we make use of a well known
group theoretical lemma:

\begin{lemma}\label{L2.13}
Let $G$ be a group and assume that $H$ is a subgroup of
finite index in $G$. If $f$ is a homomorphism from $G$ to
another group, then
$$ (G:H) = (G^f:H^f)(G_fH:H), $$
where $G^f = \im f$, $G_f = \ker f$, and $H^f$ is the image
of the restriction of $f$ to $H$.
\end{lemma}

We apply this lemma to $G = E_i$, $H = E_i^*$, and $f = N_i$.
Then $G_f = \{\eps \in E_i \sth N_i \eps = 1\} \subseteq E_i^* = H$,
$G^f = E_i^N = \{N_i \eps \sth \eps \in E_i\}$, and
$H^f = E_k^2$; now Lemma \ref{L2.13} gives
$$ (E_i:E_i^*) = (G:H) = (G^f:H^f) = (E_i^N:E_k^2).$$
Putting (\ref{E2.10})--(\ref{E2.12}) together, we find
\begin{eqnarray*}
(R:R_\pi) & = & 2^{t+\kappa-\lambda-2} q(K) h_k (E_1E_2E_3:E^*) \\
  & = & 2^{t+\kappa-\lambda-2} q(K) h_k
         \frac{(E_1E_2E_3:E_1^*E_2^*E_3^*)}{(E^*:E_1^*E_2^*E_3^*)} \\
  & = & 2^{t+\kappa-\lambda-2-v} q(K) h_k
           \prod  \frac{(E_i^N:E_k^2)}{(H_0:E_k^2)},
\end{eqnarray*}
which is (\ref{E2.5}).

The only claim left to prove is (\ref{E2.7}). If $\fp$ is a place
in $k$ which ramifies in $K/k$, then $e(\fp) = 2$ if $\fp$ ramifies
in two of the three intermediate fields, and $e(\fp) = 4$ if $\fp$
is ramified in $k_i/k$ for $i = 1, 2, 3$ (this can only happen
for $\fp \mid 2$). This observation yields the first and the
second equation in (\ref{E2.7}).

Npw $n = (k:\Q) = r_k + 2s_k$ and $4n = (K:\Q) = r_K + 2s_K$,
where $r_*$ (resp. $s_*$) denotes the number of real (resp.
complex) infinite places in a field. Suppose that exactly
$d$ infinite places of $k$ ramify in $K/k$; then
$r_K = 4(r_k-d)$, $s_K = 4s_k + 2d$, and Dirichlet's unit
theorem gives $\kappa = r_k+s_k-1$ and
$$ \lambda = r_K + s_K - 1 =
  4(r_k-d) + 4s_k + 2d - 1 = 4\kappa - 2d + 3. $$

\section{Walter's Formula}
Assume that $K/k$ is a normal extension, $\Gal(K/k) = (\Z/l\Z)^m$
($l$ prime), and suppose moreover that there is no ramification
at the infinite primes of $k$. The formula given by Kuroda
\cite{Kur2} is
$$ \frac{H}h = l^{-A} (E : E_\Omega) \cdot \prod \frac{h_i}h. $$
here
\begin{itemize}
\item $h$ is the class number of $k$,
\item $H$ is the class number of $K$,
\item $h_i$ is the class number of the intermediate field $k_i$;
      there are exactly $l = \frac{l^m-1}{l-1}$ such fields $k_i$;
\item $E$ is the unit group of $\OO_K$,
\item $E_\Omega = \prod E_i$ is the group generated by the units of
      the subfields $k_i$,
\item $A = \frac{l^u-1}{l-1} - u + \frac{\kappa+1}2
      \Big( (m-1)(l^m-1) + \frac{l^m-1}{l-1} - m\Big)
      - \kappa\Big( \frac{l^m-1}{l-1} - m\Big)$;
\item $u$ is the number of independent extensions of type
      $k(\sqrt[l]{\eps}\,)$, where $\eps$ is a unit in $\OO_k$.
\end{itemize}
Using these notations, the formula given by Walter \cite{Wal}
reads as follows:
$$ \frac{H}h = l^{-A} (E:WE_\Omega) \cdot \prod \frac{h_i}h,$$
where $W$ is the group of roots of unity in $K$ and
$$ A = \frac{l^u-1}{l-1} - u + \frac12(m-1)(\lambda-1)
        - \frac{\kappa-1}2\Big(\frac{l^m-1}{l-1} - 1\Big) - w.$$
In order to define $w$, we have to distinguish two cases:

\medskip \noindent
(A) None of the $k_i$ has the form $k_i = k(\sqrt{-1}\,)$:
    then $w = 0$;

\medskip \noindent
(B) $l=2$ and $k_1 = k(\sqrt{-1}\,)$, say; then
    $2^w = (W^{(2)}:W_1^{(2)})$, where $W^{(2)}$
    (resp. $W_1^{(2)}$) is the $2$-Sylow subgroup of $W$
    (resp. $W_1$), and $W_1$ is the group of roots of
    unity in $k_1$.
\medskip

It is easily seen that $2^w = (W : \prod W_i)$ (just remember
that the field of $p^n$-th roots of unity has cyclic Galois
group over $\Q$ for $p > 2$). If we recall the fact that
Kuroda's formula applies only if no infinite places ramify
(which implies that $\lambda+1 = l^m(\kappa+1)$), the two
formulas give the same result if and only if
$\gamma := (E : E_\Omega)/2^w(E:WE_\Omega) = 1$. Obviously
$\gamma = 1$ if $l > 2$; for $l = 2$ we obtain
$$ (E:E_\Omega) = (E : WE_\Omega)(W E_\Omega : E_\Omega)
   = (E : WE_\Omega)(W : E_\Omega \cap W). $$
Now $\prod W_i \subset E_\Omega \cap W$, therefore
$$ (W : E_\Omega \cap W) =
     \frac{(W : \prod W_i)}{(E_\Omega  \cap W: \prod W_i)}$$
and
$$ \gamma = \frac{(E : E_\Omega)}{2^w (E : WE_\Omega)}
      = (E_\Omega \cap W : \prod W_i). $$
As can be seen, $\gamma = 1$ if and only if
$W \cap \prod E_i = \prod W_i$, i.e., if and only if every
root of unity that can be written as a product of units from
the subfields is actually a product of roots of unity lying
in the subfields. If $K$ does not contain the $8$th roots of
unity, this is certainly true; the following example shows that
it does not hold in general.

Take $k = \Q(\sqrt{3}\,)$, $K = \Q(i,\sqrt{2}, \sqrt{3}\,) =
\Q(\zeta_{24})$; Walter's formula yields $h(K) = 2$; but
$\Z[\zeta_{24}]$ is known to be Euclidean with respect to the norm,
and therefore has class number $1$.

Put $k_1 = k(i)$, $k_2 = k(\sqrt{2}\,)$, $k_3 = k(\sqrt{-2}\,)$,
and set
$$ \eps_2 = 1 + \sqrt{2}, \quad \eps_3 = 2 + \sqrt{3}, \quad
   \eps_6 = 5 + 2\sqrt{6}, $$
$$ \sqrt{\eps_3} = \frac{1+\sqrt{3}}{\sqrt{2}}, \quad
   \sqrt{\eps_6} = \sqrt{2} + \sqrt{3}, $$
$$ \sqrt{-\eps_3} = \frac{1+\sqrt{3}}{i\sqrt{2}}, \quad
   \sqrt{i\eps_3} = \frac{1+\sqrt{3}}{1-i}, $$
$$ \sqrt{\zeta_8\eps_2\sqrt{\eps_3\eps_6}}
   \ = \ \frac14(4 + 3\sqrt{2} + 2\sqrt{3}
         + \sqrt{6} + 2i + \sqrt{-2} + \sqrt{-6}\,).$$
Then $\kappa = 1$, $\lambda = 3$, $t_1 = 2$, $t_2 = 3$, $d = 2$,
$u = 2$ since $k_1 = k(\sqrt{-1}\,)$ and $k_2 = k(\sqrt{\eps_3}\,)$,
$w = 1$ since $W = \la \zeta_{24} \ra$ and $W_1 = \la \zeta_{12} \ra$,
and $q(K) = 2$ (in this example, the unit indices $(E : E_\Omega)$
and $(E:WE_\Omega)$ coincide, and in Wada \cite{Wad} it is shown
that $(E : E_\Omega) = 2$). Walter's formula gives
$$ \ts h(K) = \frac12 q(K) \prod h_i = \frac12 \cdot 2 \cdot 2 = 2.$$

I have also computed the groups that occur in the proof of
Kuroda's formula:
\begin{itemize}
\item $E_k = \la -1, \eps_3 \ra$;
\item $E_1 = \la \zeta_{12}, \sqrt{i\eps_3}\,\ra$, \
      $E_1^* = \la \zeta_{12}, \eps_3\ra$, \
      $E_1^N = \la \eps_3 \ra$;
\item $E_2 = \la -1, \eps_2, \sqrt{\eps_3}, \sqrt{\eps_6} \,\ra$, \
      $E_2^* = \la -1, \eps_2^2, \eps_3, \sqrt{\eps_6}\,\ra$, \
      $E_2^N = \la -1, \eps_3 \, \ra$;
\item $E_3 = \la -1, \sqrt{-\eps_3}\,\ra$,
      $E_3^* = -1, \eps_3\, \ra$, $E_3^N - \la \eps_3 \,\ra$;
\item $\prod (E_i : E_i^*) = 2 \cdot 4 \cdot 2 = 16$;
\item $H_0 = \la \eps_3 \ra$; $H = H_0$ since $-1$ is not a norm
      residue at $\infty$;
\item $E = \la \zeta_{24}, \eps_2, \sqrt{\eps_3},
       \sqrt{\zeta_8\eps_2\sqrt{\eps_3\eps_6}}\,\ra$
      (see Wada \cite{Wad});
\item $E_1E_2E_3 = \la \zeta_{24}, \eps_2, \sqrt{\eps_3},
                        \sqrt{\eps_6}\,\ra$
      ($\zeta_8 = \sqrt{i\eps_3}/\sqrt{\eps_3}$\,), $q(K) = 2$;
\item $E^* = \la \zeta_{12}, \eps_2^2, \sqrt{i\eps_3},
        \sqrt{\eps_6}\,\ra$, $E:E^*) = 8$, $(E_1E_2E_3:E^*) = 4$;
\item $E_1^* E_2^* E_3^* = \la \zeta_{12}, \eps_2^2, \eps_3,
       \sqrt{\eps_6}\,\ra$, $(E^* : E_1^* E_2^* E_3^*) = 2$;
\item $E^{(2)} = \la i, \sqrt{\eps_3} \ra$;
\item $\ker \psi = \{(\overline{1}, \overline{1}, \overline{1})$,
      $(iE_1^*, \sqrt{\eps_3} E_2^*, \sqrt{-\eps_3} E_3^*)\}$,
      because $i \sqrt{\eps_3} \sqrt{-\eps_3} = \eps_3$ can be
      written in the form $\eps_3 = \eps_3 \cdot 1 \cdot 1 \in
      E_1^* E_2^* E_3^*$, while $\sqrt{\eps_3} \notin E_2^*$.
\end{itemize}
The prime ideal $\fl$ in $k_3$ above $2$ generates an ideal
class of order $2$ in $\Cl(k_3)$: $\fl$ is not principal,
because its relative norm to $\Q(\sqrt{-6}\,)$ is not, and
its order divides $2$ because $\frf^2 = (1+\sqrt{3}\,)$. This
implies
$$ \# A_1 = \# A_2 = 1, \quad
    A_3 = \la [\fl] \ra, \quad
   \ker j = \ker j^*  = 1 \times 1 \times A_3 \simeq A_3.$$

\newpage

Marcin Mazur and Stephen V. Ullom [{\em Unit indices and cohomology for
    biquadratic extensions of imaginary quadratic fields},
  Journal de Th\'eorie des Nombres de Bordeaux {\bf 20} (2008), 183--204]
observed the following error in my manuscript:
\begin{quote}
  A substantial part of Lemmermeyer's paper is devoted to his formula (2.5)
  for the index $[R : R_\pi]$. Here $R$ is the group of fractional ideals of
  $L$ of the form $I_1I_2I_3$, where $I_j$ is the image of an ambiguous
  fractional ideal of $L_i$. He defines $R_\pi$ as the principal ideals of
  $R$ but this is incorrect. It should be defined as those principal
  fractional ideals of $L$ whose square comes from a principal ideal of $M$.
  This is in fact the group Lemmermeyer uses in his computations so
  the mistake does not affect the validity of his results.
\end{quote}

\end{document}